\documentclass{siamltex}%
\usepackage{amsmath}
\usepackage{amsfonts}
\usepackage{amssymb}
\usepackage{graphicx}
\usepackage{color}%
\usepackage{fullpage}
\providecommand{\U}[1]{\protect\rule{.1in}{.1in}}

\newcommand{\field}[1]{\mathbb{#1}}
\newtheorem{algorithm}[theorem]{Algorithm}
\newtheorem{remark}[theorem]{Remark}
\begin{document}

\title{Convergence of the Square Root Ensemble Kalman Filter\\
 in the Large Ensemble Limit}
\author{Evan Kwiatkowski\footnotemark[1]
\and Jan Mandel\footnotemark[1]\ \footnotemark[2]}
\maketitle

\begin{abstract}
Ensemble filters implement sequential Bayesian estimation by representing the
probability distribution by an ensemble mean and covariance. Unbiased square
root ensemble filters use deterministic algorithms to produce an analysis
(posterior) ensemble with a prescribed mean and covariance, consistent with the
Kalman update. This includes several filters used in practice, such as the
Ensemble Transform Kalman Filter (ETKF), the Ensemble Adjustment Kalman Filter
(EAKF), and a filter by Whitaker and Hamill. We show that at every time index,
as the number of ensemble members increases to infinity, the mean and
covariance of an unbiased ensemble square root filter converge to those of the
Kalman filter, in the case of a linear model and an initial distribution of
which all moments exist. The convergence is in all $L^{p}$, $1\le p<\infty$, and the convergence
rate does not depend on the model or data dimensions. The result holds in 
infinitely dimensional Hilbert spaces as well.

\end{abstract}

\renewcommand{\thefootnote}{\fnsymbol{footnote}} \footnotetext[1]{Department
of Mathematical and Statistical Sciences, University of Colorado Denver,
Denver, CO, USA. Partially supported by the U.S. National Science Foundation
under the grant DMS-1216481.} \footnotetext[2]{Institute of Computer Science,
Academy of Sciences of the Czech Republic, Prague, Czech Republic. Partially
supported by the Czech Science Foundation under the grant
13-34856S.}


\section{Introduction}

\label{sec:introduction}

Data assimilation uses Bayesian estimation to incorporate observations into
the model of a physical system. The model produces the forecast estimate, and the incorporation of the observation data produces the analysis estimate. The resulting analysis is used to initialize the
next run of the model, producing the next forecast, which subsequently is used
in the next analysis, and the process thus continues. Data assimilation is
widely used, e.g., in geosciences \cite{Kalnay-2003-AMD}.

The Kalman filter represents probability distributions by the mean and
covariance. It is an efficient method when the probability
distributions are close to Gaussian. However, in applications, the dimension
of the state is large, and it is not feasible to even store the covariance of the system state exactly. Ensemble Kalman filters are variants of Kalman filters in
which the state probability distribution is represented by an ensemble of
states, and the state mean and covariance are estimated from the ensemble
\cite{Evensen-2009-DAE}. The dynamics of the model, which could be nonlinear
in this case, are applied to each ensemble member to produce the forecast
estimate. The simplest estimate of covariance from the ensemble is the sample
covariance, which, however, suffers from sampling errors for small ensembles.
For this reason, localization techniques, such as tapering
\cite{Furrer-2007-EHP} and covariance inflation \cite{Anderson-2007-ACI,Anderson-1999-MCI},
need to be used for small ensembles \cite[Ch.~15]{Evensen-2009-DAE}. Simulation
studies have shown that the ensemble filters with relatively small ensembles
and with localization and covariance inflation are able to efficiently handle
nonlinear dynamics and high dimension. 

The major differences between ensemble Kalman filters are in the way the
analysis ensemble is produced from the forecast and the data. The analysis
ensemble can be formed in a stochastic or deterministic manner. The purpose of
this paper is to examine ensemble Kalman filters that use a deterministic
method to produce an analysis ensemble with exactly the desired statistics.
Such filters are called unbiased square root filters, because the ensemble
mean equals the prescribed mean, and construction of the analysis ensemble to
match the prescribed ensemble covariance leads to taking the square root of a
matrix. This includes several filters used in practice, such as the Ensemble
Transform Kalman Filter (ETKF)
\cite{Bishop-2001-ASE,Hunt-2007-EDA,Wang-2004-WIB}, the Ensemble Adjustment
Kalman Filter (EAKF) \cite{Anderson-2001-EAK}, and a filter by Whitaker and
Hamill \cite{Whitaker-2002-EDA}. Criteria necessary for an ensemble square
root filter to be unbiased are discussed in
\cite{Livings-2008-UES,Sakov-2008-IFE,Tippett-2003-ESR}. The base square root
ensemble Kalman filter (cf., Algorithm~\ref{alg:srf} below) is often modified to support 
localization and covariance inflation for small ensembles and nonlinear problems. 
Since we are interested in large ensemble asymptotics and the linear case, 
localization and covariance inflation are not studied in this paper.

An important question for understanding ensemble filters is a law of large
numbers as the size of the ensemble grows to infinity, even if practical
implementations are necessarily limited to small ensembles. In
\cite{LeGland-2011-LSA,Mandel-2011-CEK}, it was proved independently for the
version of the ensemble Kalman filter with randomized data from
\cite{Burgers-1998-ASE}, that the ensemble mean and the covariance matrix
converge to those of the Kalman filter, as the number of ensemble members
grows to infinity. Both analyses obtain convergence in all $L^{p}$, $1\leq
p<\infty$, but the convergence results are not independent of the space
dimensions. The proof in \cite{Mandel-2011-CEK} relies on the fact that
ensemble members are exchangeable, uses uniform integrability, and does not
provide any convergence rates, while \cite{LeGland-2011-LSA} uses stochastic
inequalities for the random matrices and vectors to obtain convergence with
the usual Monte Carlo rate $1/\sqrt{N}$, but relies on entry-by-entry arguments.

Here, we show that at every time index, as the number of ensemble members
increases to infinity, the mean and covariance of an unbiased ensemble square
root filter converge to those of the Kalman filter, in the case of a linear
model and an initial distribution of which all moments exist. The convergence
is in all $L^{p}$, with the usual rate $1/\sqrt{N}$, the constant does not
depend on the dimension, and the result holds in the infinitely
dimensional case as well. The constants in the estimate are constructive and
depend only on the model and the data, namely the norms of the model operators
and of the inverse of the data covariance, and of the vectors given. The
analysis builds on some of the tools developed in \cite{LeGland-2011-LSA}, and
extends them to obtain bounds on the operators involved in the formulation of
square root ensemble filters, independent of the state space and data space
dimensions, including the infinitely dimensional case. The square root
ensemble filters are deterministic, which avoids technical complications
associated with data perturbation in infinite dimension. Convergence of
ensemble filters with data perturbation, independent of the dimension, will be
studied elsewhere.

The main idea of the analysis is simple: by the law of large numbers, the
ensemble mean and covariance of the initial ensemble converge to those
of the background distribution. Every analysis step is a
continuous mapping of the mean and the covariance, and the convergence in the
large ensemble limit follows. The analysis quantifies this argument.

Since the principal observation is that the mean and the covariance are
transformed in exactly the same way as in the Kalman filter, the continuity
estimates in this paper can also be interpreted as a type of stability of the
Kalman filter for arbitrary but fixed time with respect to perturbations of
the initial distribution. The continuity of the Kalman filter formulas is 
foundational, and, as such, has not received much attention. The
pointwise estimates in the present paper are somewhat stronger, and they imply
local Lipschitz continuity with polynomial growth of the constant. The 
$L^{p}$ estimates can be interpreted as the stability of each Kalman filter
time step separately with respect to random perturbations of the initial mean
vector and covariance operator. This type of stability of the Kalman filter
seems not to have been studied before. However, there has been a keen interest
in long-time stability, particularly the effect of a finite error in the
initial distribution diminishing over time, both theoretically for the Kalman
filter (e.g., \cite{LeBreton-2000-AOA,Ocone-1996-ASO,Zhang-1991-LSE}) and
empirically for the ensemble Kalman filter in meteorology (e.g.,
\cite{Zhang-2004-IIE}). For a class of abstract dynamical systems with the
whole state observed, the ensemble Kalman filter with a fixed ensemble size
was proved to be well-posed and not blowing up faster than exponentially, and
to stay within a bounded distance from the truth if sufficiently large
covariance inflation is used \cite{Kelly-2014-WAE}. 

The paper is organized as follows: in Section \ref{sec:notation} we introduce notation and review select background concepts. Section \ref{sec:algorithms}
contains statements of the Kalman filter and the unbiased square root filter,
and shows some basic properties, which are needed later. In Section
\ref{sec:continuity}, we show the continuity properties of the transformation of
the statistics from one time step to the next. Section \ref{sec:LpLLN}
contains a derivation of the $L^{p}$ laws of large numbers, needed for the
convergence of the statistics of the initial ensemble. Section
\ref{sec:convergence} presents the main result. 

\section{Notation and preliminaries}

\label{sec:notation}We will work with random elements with values in a Hilbert
space $V$. Readers interested in finite dimension only can think of random
vectors in $V=\mathbb{R}^{n}$. All notations and proofs are presented in a way
that applies in $\mathbb{R}^{n}$, as well as in a more general Hilbert space.

\subsection{Finite dimensional case}

Vectors $u\in\mathbb{R}^{n}$ are columns, and the inner product is
$\left\langle u,v\right\rangle =u^{\ast}v$ where $u^{\ast}$ denotes transpose, and $|u|=\left\langle u,u\right\rangle ^{1/2}$ is the vector norm of $u$. Throughout this paper, we will use single bars $|\cdot|$ for deterministic norms and double bars $\|\cdot\|$ for stochastic norms.
We will use the notation $\left[  V\right]  $ for the space of all $n\times n$
matrices. For a matrix $M$, $M^{\ast}$ denotes transpose, and $\left\vert
M\right\vert $ stands for the spectral norm. We will also need the
Hilbert-Schmidt norm (more commonly called Frobenius norm) of matrices,
induced by the corresponding inner product of two $n\times n$ matrices,%
\[
|A|_{\mathrm{HS}}=\left(  \sum_{i,j=1}^{n}a_{ij}^{2}\right)  ^{1/2}%
=\left\langle A,A\right\rangle _{\mathrm{HS}}^{1/2},\quad\left\langle
A,B\right\rangle _{\mathrm{HS}}=\sum_{i,j=1}^{n}a_{ij}b_{ij}.
\]
The Hilbert-Schmidt norm dominates the spectral norm,%
\begin{equation}
\left\vert A\right\vert \leq\left\vert A\right\vert _{\mathrm{HS}},
\label{eq:hilbert-schmidt-dominates}%
\end{equation}
for any matrix $A$. The notation $A\geq0$ means that $A$ is symmetric and
positive semidefinite, $A>0$ means symmetric positive definite. For $Q\geq0$, $X\sim
N\left(  u,Q\right)  $ means that random vector $X$ has the normal
distribution on $\mathbb{R}^{n}$, with mean $u$ and covariance matrix $Q$. For
vectors $u,v\in\mathbb{R}^{n}$, their tensor product is the $n\times n$
matrix,
\[
u\otimes v=uv^{\ast}\in\left[  V\right]  .
\]
It is evident that%
\begin{equation}
|u\otimes v|_{\mathrm{HS}}=|u||v|, \label{eq:hilbert-schmidt-tensor}%
\end{equation}
because%
\[
|u\otimes v|_{HS}^{2}=\sum_{i=1}^{n}\sum_{j=1}^{n}\left(  u_{i}v_{j}\right)
^{2}=\sum_{i=1}^{n}u_{i}^{2}\sum_{j=1}^{n}v_{j}^{2}=|u|^{2}|v|^{2}.
\]

\subsection{Hilbert space case}

\label{sec:hilbert}

Readers interested in finite dimension only should skip this section. In the
general case, $V$ is a separable Hilbert space equipped with inner product
$\left\langle u,v\right\rangle $ and the norm $\left\vert u\right\vert
=\left\langle u,u\right\rangle ^{1/2}$. The space of all bounded operators
from Hilbert space $U$ to Hilbert space $V$ is $\left[  U,V\right]  $, and
$\left[  V\right]  =[V,V]$. For a bounded linear operator $A\in\left[
U,V\right]  $, $A^{\ast}\in\left[  V,U\right]  $ denotes the adjoint operator,
and $\left\vert A\right\vert $ is the operator norm. The Hilbert-Schmidt norm
of a linear operator on $V$ is defined by
\[
|A|_{\mathrm{HS}}=\left\langle A,A\right\rangle _{\mathrm{HS}}^{1/2}%
,\quad\left\langle A,B\right\rangle _{\mathrm{HS}}=\sum_{i=1}^{\infty
}\left\langle Ae_{i},Be_{i}\right\rangle ,
\]
where $\left\{  e_{i}\right\}  $ is any complete orthonormal sequence; the
values do not depend on the choice of $\left\{  e_{i}\right\}  $. An operator
on $V$ is called a Hilbert-Schmidt operator if $\left\langle A,B\right\rangle
_{\mathrm{HS}}<\infty$, and $\mathrm{HS}\left(  V\right)  $ is the space of
all Hilbert-Schmidt operators on $V$. The Hilbert-Schmidt norm again dominates
the spectral norm, so (\ref{eq:hilbert-schmidt-dominates}) holds, and
$\mathrm{HS}\left(  V\right)  \subset\left[  V\right]  $. The importance of
$\mathrm{HS}\left(  V\right)  $ for us lies in the fact that $\mathrm{HS}%
\left(  V\right)  $ is a Hilbert space, while $\left[  V\right]  $ is not. 

The
notation $A\geq0$ for operator $A\in\left[  V\right]  $ means that $A$
is symmetric, $A=A^{\ast}$, and positive semidefinite, $\left\langle
Au,u\right\rangle \geq0$ for all $u\in V$. The notation $A>0$ means here that
$A$ is symmetric and \emph{bounded below}, that is, $\left\langle
Au,u\right\rangle \geq\alpha\left\vert u\right\vert ^{2}$ for all $u\in V$ and
some $\alpha>0.$ In particular, if $A>0$, then $A^{-1}\in\left[  V\right]  $. For $A$, $B\in\left[  V\right]  ,$ $A\leq B$ means that $A$ and $B$ are
symmetric, and $B-A\geq0.$ It is well known from spectral theory that
$0\leq A\leq B$ implies $|A|\leq|B|$. Let $\rho(A)$ denote the spectral radius of $A$, and if $A$ is symmetric, then $\rho(A)=|A|$. For a symmetric operator $A\ge 0$, there exists a 
unique symmetric $A^{1/2}\ge 0$ such that $A^{1/2}A^{1/2}=A$, and if $A>0$, then also $A^{1/2}>0$ and $A^{-1/2}=(A^{1/2})^{-1}\in\left[  V\right]$.

An operator $A\geq0$ is trace class if $\operatorname*{Tr}A<\infty$, where
$\operatorname*{Tr}A$ is the trace of $A$, defined by
\[
\operatorname*{Tr}A=\sum_{i=1}^{\infty}\left\langle Ae_{i},e_{i}\right\rangle
.
\]
If $A\geq0$ is trace class, then $A$ is Hilbert-Schmidt, because $\left\vert
A\right\vert _{\mathrm{HS}}\leq\operatorname*{Tr}A$.

For vectors $u,v\in V$, their tensor product $u\otimes v\in\left[  V\right]  $
is now a mapping defined by%
\[
u\otimes v:w\in V\mapsto u\left\langle v,w\right\rangle ,
\]
and the proof of (\ref{eq:hilbert-schmidt-tensor}) becomes%
\begin{align*}
\left\vert x\otimes y\right\vert _{HS }^{2}  &  =\sum
_{i=1}^{\infty}\left\vert \left(  x\otimes y\right)  e_{i}\right\vert
^{2}=\sum_{i=1}^{\infty}\left\vert x\left\langle y,e_{i}\right\rangle
\right\vert ^{2}  =\left\vert x\right\vert ^{2}\sum_{i=1}^{\infty}\left\vert \left\langle
y,e_{i}\right\rangle \right\vert ^{2}=\left\vert x\right\vert ^{2}\left\vert
y\right\vert ^{2},
\end{align*}
from Bessel's equality.

The mean of a random element $E\left(  X\right)  \in V$ is defined by%
\[
\left\langle E\left(  X\right)  ,y\right\rangle =E\left(  \left\langle
X,y\right\rangle \right)  \quad\text{for all }y\in V.
\]
The mean of the tensor product of two random elements is defined by%
\[
\left\langle E\left(  X\otimes Y\right)  w,y\right\rangle =E(\left\langle
X,y\right\rangle \left\langle Y,w\right\rangle )\quad\text{ for all }w,y\in
V.
\]
The covariance of a
random element $X$ (defined below in (\ref{eq:def-covariance})) exists when the second moment $E(\left\vert X\right\vert
^{2})<\infty$, and the proposed covariance is a trace class operator. On the other
hand, if $Q\geq0$ is trace class, the normal distribution $N\left(
u,Q\right)  $ can be defined as a probability measure on $V$, consistently
with the finite dimensional case. Just as in the finite dimensional case, if
$X\sim N\left(  u,Q\right)  $, then $X$ has all moments $E\left(  \left\vert
X\right\vert ^{p}\right)  <\infty$, $1\leq p<\infty$. See
\cite{DaPrato-2006-IIA,Kuo-1975-GMB,Lax-2002-FA} for further details.

\subsection{Common definitions and properties}

The rest of the background material is the same regardless if $V=\mathbb{R}^{n}$ or if $V$ is a general
Hilbert space. To unify the nomenclature, matrices are identified with the
corresponding operators of matrix-vector multiplication. Covariance of a
random vector $X\in V$ is defined by%
\begin{align}
\operatorname{Cov}(X)  &  =E(\left(  X-E\left(  X\right)  \right)
\otimes\left(  Y-E\left(  Y\right)  \right)  )\nonumber\\
&  =E(X\otimes Y)-E(X)\otimes E(Y), \label{eq:def-covariance}%
\end{align}
if it exists. For $1\leq p<\infty$, the space of all random elements $X\in V$ with finite moment $E\left(  \left\vert X\right\vert ^{p}\right)
<\infty$ is denoted by $L^{p}\left(  V\right)  $, and it is equipped with the
norm $\Vert X\Vert_{p}=(E(|X|^{p}))^{1/p}.$ If $1\leq p\leq q<\infty$ and
$X\in L^{q}(V)$, then $X\in L^{p}(V)$ and $\left\Vert X\right\Vert _{p}%
\leq\left\Vert X\right\Vert _{q}$. If $X,Y\in L^{2}(V)$ are independent, then
$\operatorname{Cov}(X,Y)=0$ and $E\left(  \left\langle X,Y\right\rangle
\right)  =0$. The following lemma will be used repeatedly in obtaining $L^{p}$
estimates. It is characteristic of our approach to use higher order
norms to bound lower order norms.

\begin{lemma}
[$L^{p}$ Cauchy-Schwarz inequality]\label{lem:lp-cauchy}If $U,V\in L^{2p}(V)$
and $1\leq p<\infty$, then
\[
\Vert|U||V|\Vert_{p}\leq\Vert U\Vert_{2p}\Vert V\Vert_{2p}.
\]

\end{lemma}

\begin{proof}
By the Cauchy-Schwartz inequality in $L^{2}(V),$ applied to the random
variables $|U|^{p}$ and $|V|^{p}$, which are in $L^{2}(V)$,%
\begin{align*}
\Vert|U||V|\Vert_{p}^{p}  &  =E\left(  |U|^{p}|V|^{p}\right)  \leq\left(
E|U|^{2p}\right)  ^{1/2}\left(  E|V|^{2p}\right)  ^{1/2}\\
&  =\left(  \left(  E|U|^{2p}\right)  ^{1/2p}\left(  E|V|^{2p}\right)
^{1/2p}\right)  ^{p}=\Vert X\Vert_{2p}^{p}\Vert Y\Vert_{2p}^{p}.
\end{align*}
Taking the $p$-th root of both sides yields the desired inequality.
\end{proof}

An ensemble $\boldsymbol{X}_{N}$ consists of random elements $X_{i}$,
$i=1,\ldots,N$. The ensemble mean is denoted by $\overline{\boldsymbol{X}}%
_{N}$ or $E_{N}(\boldsymbol{X}_{N})$, and is defined by%
\begin{equation}
\overline{\boldsymbol{X}}_{N}=E_{N}(\boldsymbol{X}_{N})=\frac{1}{N}\sum
_{i=1}^{N}X_{i}.\label{eq:def-ens-mean}%
\end{equation}
The ensemble covariance is denoted by $Q_{N}$ or $C_{N}\left(  \boldsymbol{X}%
_{N}\right)  $, and is defined by%
\begin{align}
Q_{N}=C_{N}\left(  \boldsymbol{X}_{N}\right)   &  =\frac{1}{N}\sum_{i=1}%
^{N}(X_{i}-\overline{\boldsymbol{X}}_{N})\otimes(X_{i}-\overline
{\boldsymbol{X}}_{N})\nonumber\\
&  =E_{N}\left(  \boldsymbol{X}_{N}\otimes\boldsymbol{X}_{N}\right)
-E_{N}\left(  \boldsymbol{X}_{N}\right)  \otimes E_{N}\left(  \boldsymbol{X}%
_{N}\right)  ,\label{eq:def-ens-covariance}%
\end{align}
where 
\[
\boldsymbol{X}_{N}\otimes\boldsymbol{X}_{N}=X_{1}\otimes
X_{1}+\ldots+X_{N}\otimes X_{N}.
\]
We use $N$ instead of the more
common $N-1$, which would give an unbiased estimate, because it 
allows writing the sample covariance with the sample mean in (\ref{eq:def-ens-covariance})
without the additional multiplicative factors $N/(N-1)$.
Note that convergence as $N\rightarrow\infty$ is not affected.

\section{Definitions and basic properties of the algorithms}

\label{sec:algorithms} In the Kalman filter, the probability distribution of
the state of the system is described by its mean and covariance.
We first consider the analysis step, which uses Bayes' theorem to incorporate
an observation into the forecast state and covariance to produce the analysis
state and covariance. The system state $X$ is an $\mathbb{R}^{n}$-valued
random vector. We denote by $\overline{X}^{\mathrm{f}}\in\mathbb{R}^{n}$ the
forecast state mean, $Q^{\mathrm{f}}\in\mathbb{R}^{n\times n}$ the forecast
state covariance, and $\overline{X}^{\mathrm{a}}\in\mathbb{R}^{n}$ and
$Q^{\mathrm{a}}\in\mathbb{R}^{n\times n}$ the analysis state mean and
covariance, respectively. The observation data vector is $d\in\mathbb{R}^{m}$,
where $d-HX\sim N\left(  0,R\right)  $, with $H\in\mathbb{R}^{m\times n}$ the
linear observation operator, and $R\in\mathbb{R}^{m\times m}$, $R>0$, is the observation
error covariance. 

Given the forecast mean and covariance, the 
Kalman Filter analysis mean
and covariance are%
\begin{align}
\overline{X}^{\mathrm{a}}  &  =\overline{X}^{\mathrm{f}}+K(d-H\overline
{X}^{\mathrm{f}})\label{eq:filtering-analysis-mean}\\
Q^{\mathrm{a}}  &  =Q^{\mathrm{f}}-Q^{\mathrm{f}}H^{\ast}(HQ^{\mathrm{f}%
}H^{\ast}+R)^{-1}HQ^{\mathrm{f}}=Q^{\mathrm{f}}-KHQ^{\mathrm{f}},
\label{eq:filtering-analysis-covariance}%
\end{align}
where
\begin{equation}
K=Q^{\mathrm{f}}H^{\ast}(HQ^{\mathrm{f}}H^{\ast}+R)^{-1}
\label{eq:kalman-gain}%
\end{equation}
is the Kalman gain matrix. See, e.g.,
\cite{Anderson-1979-OF,Jazwinski-1970-SPF,Simon-2006-OSE}. 

In the general case, the state space $\mathbb{R}^{n}$ and the data space $\mathbb{R}^{m}$
above become separable Hilbert spaces, one or both of which may be infinitely dimensional,
and matrices become bounded linear operators. The Kalman filter 
formulas (\ref{eq:filtering-analysis-mean})--(\ref{eq:kalman-gain}) remain the same.
The assumption $R>0$ guarantees that the inverse  in $(\ref{eq:kalman-gain})$ 
is well-defined. In particular, when the data space is infinitely dimensional, the definition of a positive
definite operator in Section~\ref{sec:hilbert} implies that this inverse is bounded. In that case, $R$ 
cannot be the covariance of a probability distribution in the classical sense, because $R>0$ cannot be of trace class
in infinite dimension.
However, all statistical estimates will be in the state space, not the data space, so this is not a problem.
 
For future
reference, we introduce the operators
\begin{align}
\mathcal{K}(Q)  &  =QH^{\ast}(HQH^{\ast}+R)^{-1},\label{eq:def-operator-K}\\
\mathcal{B}(X,Q)  &  =X+\mathcal{K}(Q)(d-HX),\label{eq:def-operator-B}\\
\mathcal{A}\left(  Q\right)   &  =Q-QH^{\ast}(HQH^{\ast}+R)^{-1}%
HQ\label{eq:def-operator-A-without-kalman-gain}\\
&  =Q-\mathcal{K}(Q)HQ, \label{eq:def-operator-A}%
\end{align}
which evaluate the Kalman gain, analysis mean, and the analysis covariance,
respectively, in the Kalman filter equations (\ref{eq:filtering-analysis-mean}%
)--(\ref{eq:kalman-gain}). 
We are now ready to state the complete Kalman
filter for reference. The superscript $^{\left(  k\right)  }$ signifies
quantities at time step $k$.

\begin{algorithm}
[Kalman filter]\label{alg:kf}
Suppose that the model $\mathcal{M}^{(k)}$ at each time $k\geq1$ is linear,
$\mathcal{M}^{(k)}\left(  X\right)  =M^{(k)}X+b^{\left(  k\right)  }$, and the
initial mean $\overline{X}^{\left(  0\right)  }=\overline{X}^{(0),\mathrm{a}}$
and the background covariance $B=Q^{(0),\mathrm{a}}$ of the state are given.
At time $k$, the analysis mean and covariance from the previous time $k-1$ are
advanced by the model,%
\begin{align}
\overline{X}^{(k),\mathrm{f}} &  =M^{\left(  k\right)  }\overline
{X}^{(k-1),\mathrm{a}}+b^{\left(  k\right)  }%
,\label{eq:filtering-forecast-mean}\\
Q^{(k),\mathrm{f}} &  =M^{\left(  k\right)  }Q^{(k-1),\mathrm{a}}M^{\left(
k\right)  ^{\ast}}.\label{eq:filtering-forecast-covariance}%
\end{align}
The analysis step incorporates the observation $d^{\left(  k\right)  }$,
where $d^{\left(  k\right)  }-H^{\left(  k\right)  }X^{(k),\mathrm{f}}$ has mean zero and covariance $R^{\left(  k\right)  }$, and it gives the
analysis mean and covariance%
\begin{align}
\overline{X}^{(k),\mathrm{a}} &  =\mathcal{B}%
(\overline{X}^{(k),\mathrm{f}},Q^{(k),\mathrm{f}}),\label{eq:kf-mean}\\
Q^{(k),\mathrm{a}} &  =\mathcal{A}(Q^{(k),\mathrm{f}%
}),\label{eq:kf-covariance}%
\end{align}
where $\mathcal{B}$ and $\mathcal{A}$ are defined by (\ref{eq:def-operator-B}) and (\ref{eq:def-operator-A})
respectively, with $d,$ $H,$ and $R$, at time $k$.
\end{algorithm}


In many applications, the state dimension $n$ of the system is large and
computing or even storing the exact covariance of the system state is
computationally impractical. Ensemble Kalman filters address this concern.
Ensemble Kalman filters use a collection of state vectors, called an ensemble,
to represent the distribution of the system state. This ensemble will be
denoted by $\boldsymbol{X}_{N}$, comprised of $N$ random elements $X_{i}\in
\field{R}^n$, $i=1,\ldots N$. The ensemble mean and ensemble covariance, defined by
(\ref{eq:def-ens-mean}) and (\ref{eq:def-ens-covariance}), are denoted by
$\overline{\boldsymbol{X}}_{N}$ and $Q_{N}$ respectively, while the Kalman
Filter mean and covariance are denoted without subscripts, as $\overline{X}$
and $Q$.

There are several ways to produce an analysis ensemble corresponding to the
Kalman filter algorithm. Unbiased ensemble square root filters produce an
analysis ensemble $\boldsymbol{X}_{N}^{\mathrm{a}}$ such that
(\ref{eq:kf-mean}) and (\ref{eq:kf-covariance}) are satisfied for the ensemble
mean and covariance. This was not always
the case in early variants of ensemble square root filters
\cite{Sakov-2008-IFE}.
Because our results assume a linear model, we formulate the algorithm in the
linear case to cut down on additional notation.

\begin{algorithm}
[Unbiased square root ensemble filter]\label{alg:srf}
Generate the initial ensemble $\boldsymbol{X}_{N}^{(0),\mathrm{a}}$ at time $k = 0$
by sampling from a distribution with
mean $\overline{X}^{\left(  0\right)  }=\overline{X}^{(0),\mathrm{a}}$ and covariance $B$.
At time
$k\geq1$, the analysis ensemble members $X_{i}^{(k-1),\mathrm{a}}$ are advanced by the
model
\begin{equation}
X_{i}^{(k),\mathrm{f}}=M^{(k)}(X_{i}^{(k-1),\mathrm{a}})+b^{\left(  k\right)
},\text{\quad}i=1,\ldots,N, \label{eq:srf-forecast}%
\end{equation}
resulting in the forecast ensemble $\boldsymbol{X}^{(k),\mathrm{f}}$ with
ensemble mean and covariance%
\begin{equation}
\overline{\boldsymbol{X}}_{N}^{(k),\mathrm{f}}=E_{N}(\boldsymbol{X}%
_{N}^{(k),\mathrm{f}}),\quad Q_{N}^{(k),\mathrm{f}}=C_{N}(\boldsymbol{X}%
_{N}^{(k),\mathrm{f}}). \label{eq:srf-forecast-covariance}%
\end{equation}
The analysis step creates an ensemble $\overline{\boldsymbol{X}}%
_{N}^{(k),\mathrm{a}}$ which incorporates the observation that $d^{\left(
k\right)  }-H^{\left(  k\right)  }\overline{\boldsymbol{X}}_{N}^{(k),\mathrm{f}}$ has
mean zero and covariance $R^{\left(  k\right)  }$. The analysis ensemble
is constructed (in a manner determined by the specific method) to have
the ensemble mean and covariance given by
\begin{align}
E_{N}(\overline{\boldsymbol{X}}_{N}^{(k),\mathrm{a}})  &  =\mathcal{B}(\overline{\boldsymbol{X}}_{N}^{(k),\mathrm{f}}%
,Q_{N}^{(k),\mathrm{f}})=\overline{\boldsymbol{X}}_{N}^{(k),\mathrm{f}%
}+\mathcal{K}(Q_{N}^{(k),\mathrm{f}})(d^{(k)}-H(\overline{\boldsymbol{X}}%
_{N}^{(k),\mathrm{f}})),\label{eq:srf-mean}\\
C_{N}(\overline{\boldsymbol{X}}_{N}^{(k),\mathrm{a}})  &  =\mathcal{A}(Q_{N}^{(k),\mathrm{f}})=Q_{N}^{(k),\mathrm{f}%
}-\mathcal{K(}Q_{N}^{(k),\mathrm{f}})H^{(k)}Q_{N}^{(k),\mathrm{f}},
\label{eq:srf-covariance}%
\end{align}
where $\mathcal{B}$ and $\mathcal{A}$ are defined by (\ref{eq:def-operator-B}) and (\ref{eq:def-operator-A}) with
$d$, $H$, and $R$, at time $k$.
\end{algorithm}

The Ensemble Adjustment Kalman Filter (EAKF, \cite{Anderson-2001-EAK}), the
filter by Whitaker and Hamill \cite{Whitaker-2002-EDA}, and the Ensemble
Transform Kalman Filter (ETKF, \cite{Bishop-2001-ASE}) and its variants, the
Local Ensemble Transform Kalman Filter (LETKF, \cite{Hunt-2007-EDA}) and its
revision \cite{Wang-2004-WIB}, satisfy properties (\ref{eq:srf-mean}) and
(\ref{eq:srf-covariance}), cf., \cite{Livings-2008-UES,Tippett-2003-ESR}, and
therefore are of this form.

The background covariance $B$ does not need to be stored explicitly. Rather,
it is used in a factored form \cite{Buehner-2005-ESF,Fisher-1995-ECM},%
\begin{equation}
B=LL^{\mathrm{T}}. \label{eq:factored-B}%
\end{equation}
and only multiplication of $L$ times a vector is needed,%
\[
X_{i}^{\left(  0\right)  }=X^{\left(  0\right)  }+LY_{i},\quad Y_{i}\sim
N\left(  0,I\right)  .
\]

For example, a sample covariance of the form (\ref{eq:factored-B}) can be
created from historical data. To better represent the dominant part of the
background covariance, $L$ can consist of approximate eigenvectors for the
largest eigenvalues, obtained by a variant of the power method
\cite{Kalnay-2003-AMD}. Another popular choice is $L=TS$, where $T$ is a
transformation, such as FFT or wavelet transform, which requires no storage of
the matrix entries at all, and $S$ is a sparse matrix
\cite{Deckmyn-2005-WAR,Mirouze-2010-RCF,Pannekoucke-2007-FPW}. The covariance
of the Fourier transform of a stationary random field is diagonal, so even an
approximation with diagonal matrix $S$ is useful, and computationally very inexpensive.

A key observation of our analysis is that, in the linear case, the
transformations of the ensemble mean and covariance in unbiased square root
ensemble filters are exactly the same as in the Kalman filter, and they can be
described without reference to the ensemble at all, as shown in the next
lemma. All that is needed for the convergence of the unbiased square
root ensemble filter is the convergence of the initial ensemble mean and
covariance, and the continuity of the transformations, in a suitable
statistical sense.

\begin{lemma}
\label{lem:srf-mean-covariance-evolution}The transformations of the mean and
covariance in step $k$ of Algorithms \ref{alg:kf} and \ref{alg:srf} are the
same,%
\[%
\begin{tabular}
[c]{c|c}%
Kalman Filter &
\begin{tabular}
[c]{c}%
Unbiased Square Root\\
Ensemble Kalman Filter
\end{tabular}
\\
& \\
$\overline{X}^{(k),\mathrm{f}}=M^{\left(  k\right)  }\overline{X}%
^{(k-1),\mathrm{a}}+b^{\left(  k\right)  }$ & $\overline{\boldsymbol{X}}_{N}%
^{(k),\mathrm{f}}=M^{\left(  k\right)  }\overline{\boldsymbol{X}}_{N}^{(k-1),\mathrm{a}%
}+b^{\left(  k\right)  }$\\
$Q^{(k),\mathrm{f}}=M^{\left(  k\right)  }Q^{(k-1),\mathrm{a}}M^{\left(
k\right)  \ast}$ & $Q_{N}^{(k),\mathrm{f}}=M^{\left(  k\right)  }%
Q_{N}^{(k-1),\mathrm{a}}M^{\left(  k\right)  \ast}$\\
$\overline{X}^{(k),\mathrm{a}}=\mathcal{B}(\overline
{X}^{(k),\mathrm{f}},Q^{(k),\mathrm{f}})$ & $\overline{\boldsymbol{X}}_{N}^{(k),\mathrm{a}}=\mathcal{B}(\overline{\boldsymbol{X}}_{N}^{(k),\mathrm{f}},Q_{N}^{(k),\mathrm{f}%
})$\\
$Q^{(k),\mathrm{a}}=\mathcal{A}(Q^{(k),\mathrm{f}})$ &
$Q_{N}^{(k),\mathrm{a}}=\mathcal{A}(Q_{N}^{(k),\mathrm{f}%
})$%
\end{tabular}
\]

\end{lemma}

\begin{proof}
From (\ref{eq:srf-forecast-covariance}) and the definition
(\ref{eq:def-ens-covariance}) of ensemble covariance, it follows that
\begin{equation}
Q_{N}^{(k),\mathrm{f}}=M^{\left(  k\right)  }Q_{N}^{(k-1),\mathrm{a}%
}M^{\left(  k\right)  ^{\ast}}. \label{eq:forecast-ens-covariance}%
\end{equation}
The rest of the transformations in Algorithms \ref{alg:kf} and \ref{alg:srf}
are already the same.
\end{proof}

\section{Continuity of the analysis step}

\label{sec:continuity}

Fundamental to our analysis are continuity estimates for the operators
$\mathcal{A}$ and $\mathcal{B}$, which bring the forecast statistics to the
analysis statistics. Our general strategy is to first derive pointwise
estimates which apply to every realization of the random elements separately,
then integrate them to get the corresponding $L^{p}$
estimates. We will prove the following estimates for general covariances $Q$
and $P$, with the state covariance and sample covariance of the filters in mind. Likewise, estimates will be made with general elements $X$ and $Y$, with the state mean and sample mean of the filters in mind.

\subsection{Pointwise bounds}

The first estimate is the continuity of the Kalman gain matrix (or operator)
(\ref{eq:kalman-gain}) as a function of the forecast covariance, shown in the
next lemma and its corollary. See also \cite[Proposition 22.2]%
{LeGland-2011-LSA}.

\begin{lemma}
\label{lem:continuity-K}If $R>0$ and $P,Q\geq0,$ then%
\[
|\mathcal{K}(Q)-\mathcal{K}(P)|\leq\left\vert Q-P\right\vert \left\vert
H\right\vert |R^{-1}| \left(  1+\min\left\{  \left\vert P\right\vert
,\left\vert Q\right\vert \right\}  \left\vert H\right\vert ^{2} \left\vert R^{-1}\right\vert\right)  .
\]

\end{lemma}

\begin{proof}
Since $R>0$ and $P,Q\geq0$, $\mathcal{K}(Q)$ and $\mathcal{K}(P)$ are defined
in (\ref{eq:def-operator-K}). For $A$, $B\geq0$, we have the identity,%
\begin{equation}
(I+A)^{-1}-(I+B)^{-1}=(I+A)^{-1}(B-A)(I+B)^{-1}, \label{eq:resolvent-identity}%
\end{equation}
which is verified by multiplication by $I+A$ on the left and $I+B$ on
the right, and%
\begin{equation}
\left\vert (I+A)^{-1}-(I+B)^{-1}\right\vert \leq\left\vert B-A\right\vert ,
\label{eq:est-inv-diff}%
\end{equation}
which follows from (\ref{eq:resolvent-identity}) using the inequalities
$|(I+A)^{-1}|\leq1$, $|(I+B)^{-1}|\leq1$, because $A,B\geq0$. Now write%
\begin{equation}
(HQH^{\ast}+R)^{-1}=R^{-1/2}(R^{-1/2}HQH^{\ast}R^{-1/2}+I)^{-1}R^{-1/2},\label{eq:inverse-expansion}
\end{equation}
using the symmetric square root $R^{1/2}$ of $R$. By (\ref{eq:est-inv-diff})
with $A=R^{-1/2}HQH^{\ast}R^{-1/2}$ and $B=R^{-1/2}HPH^{\ast}R^{-1/2}$ and (\ref{eq:inverse-expansion}), we
have that%
\begin{align}
\left\vert (HQH^{\ast}+R)^{-1}-(HPH^{\ast}+R)^{-1}\right\vert  &
\leq  \left\vert Q-P\right\vert \left\vert H\right\vert ^{2}|R^{-1}|^2.
\label{eq:est-diff-HQ-HR}%
\end{align}
Since $HQH^{\ast}+R\geq R$, we have
\begin{equation}
|\left(  HQH^{\ast}+R\right)  ^{-1}|\leq|R^{-1}|. \label{eq:est-HQ-HR}%
\end{equation}
Using (\ref{eq:est-diff-HQ-HR}), (\ref{eq:est-HQ-HR}), and the definition of
the operator $\mathcal{K}$ from (\ref{eq:def-operator-K}), we have
\begin{align*}
|\mathcal{K}(Q)-\mathcal{K}(P)|=  &  |QH^{\ast}(HQH^{\ast}+R)^{-1}-PH^{\ast
}(HPH^{\ast}+R)^{-1}|\\
=  &  |QH^{\ast}(HQH^{\ast}+R)^{-1}-PH^{\ast}(HQH^{\ast}+R)^{-1}\\
&  +PH^{\ast}(HQH^{\ast}+R)^{-1}-PH^{\ast}(HPH^{\ast}+R)^{-1}|\\
\leq  &  \left\vert Q-P\right\vert \left\vert H\right\vert |R^{-1}|\left(
1+\left\vert P\right\vert \left\vert H\right\vert ^{2}\left\vert R^{-1}\right\vert \right)  .
\end{align*}
Swapping the roles $P$ and $Q$ yields 
\begin{align*}
|\mathcal{K}(Q)-\mathcal{K}%
(P)|\leq\left\vert Q-P\right\vert \left\vert H\right\vert |R^{-1}|\left(
1+\left\vert Q\right\vert \left\vert H\right\vert ^{2}\left\vert R^{-1}\right\vert \right),
\end{align*} which completes the proof.
\end{proof}

A pointwise bound on the Kalman gain follows.

\begin{corollary}
\label{cor:norm-kalman-gain}If $R>0$ and $Q\geq0$, then
\begin{equation}
|\mathcal{K}(Q)|\leq\left\vert Q\right\vert \left\vert H\right\vert | R^{-1}|.
\label{eq:norm-kalman-gain}%
\end{equation}

\end{corollary}

\begin{proof}
Use Lemma \ref{lem:continuity-K} with $P=0$.
\end{proof}

\begin{corollary}
\label{cor:norm-BX}If $R>0$ and $Q\geq0$, then
\begin{equation}
|\mathcal{B}(X,Q)|\leq|X|+|Q||H||R^{-1}|(|d-HX|). \label{eq:norm-BX}%
\end{equation}

\end{corollary}

\emph{Proof}. By the definition of operator $\mathcal{B}$ in
(\ref{eq:def-operator-B}), the pointwise bound on the Kalman gain
(\ref{eq:norm-kalman-gain}), and the triangle inequality,
\[
|\mathcal{B}(X,Q)|=|X+\mathcal{K}(Q)(d-HX)|\leq|X|+|Q||H||R^{-1}|(|d-HX|).
\qquad\endproof
\]

The pointwise continuity of operator $\mathcal{A}$ also follows from Lemma
\ref{lem:continuity-K}. To reduce the notation we introduce the constant 
\begin{equation}
c=|H|^{2}|R^{-1}|. \label{eq:def-c}%
\end{equation}
\begin{lemma}
\label{lem:continuity-A}If $R>0$ and $P,Q\geq0,$ then
\[
|\mathcal{A}(Q)-\mathcal{A}(P)|\leq|Q-P|\left(  1+c|Q|+c|P|+c^2|P||Q|\right)  .
\]

\end{lemma}

\emph{Proof}. Since $R>0$ and $P,Q\geq0$, it follows that $\mathcal{K}(Q)$,
$\mathcal{K}(P)$, $\mathcal{A}(Q)$, and $\mathcal{A}(P)$ are defined. From the
definition of operator $\mathcal{A}$ (\ref{eq:def-operator-A}), Lemma
\ref{lem:continuity-K}, and Corollary \ref{cor:norm-kalman-gain},
\begin{align*}
|\mathcal{A}(Q)-\mathcal{A}(P)|  &  =|\left(  Q-\mathcal{K}(Q)HQ\right)
-(P-\mathcal{K}(P)HP)|\\
&  =|Q-P+\mathcal{K}(P)HP-\mathcal{K}(Q)HQ|\\
&  =|Q-P+\mathcal{K}(P)HP-\mathcal{K}(P)HQ+\mathcal{K}(P)HQ-\mathcal{K}%
(Q)HQ|\\
&  \leq |Q-P| \left(  1+c|Q|+c|P|+c^2|P||Q||\right)  . \qquad\endproof
\end{align*}

\begin{remark}
The choice of $\left\vert P\right\vert $ in $\min\{|P|,|Q|\}$ in the proof
of Lemma \ref{lem:continuity-A} was made to preserve symmetry. Swapping the roles
of $P$ and $Q$ in the proof gives a sharper, but a more complicated bound
\begin{align*}
|\mathcal{A}(Q)-\mathcal{A}(P)|  &  \leq|Q-P| \left(  1+c|Q|+c|P|+c^2\cdot%
\min\left\{  \left\vert P\right\vert ^{2},\left\vert P\right\vert \left\vert
Q\right\vert ,\left\vert Q\right\vert ^{2}\right\}  \right)   .
\end{align*}

\end{remark}

Instead of setting $P=0$ above, we can get a well-known sharper pointwise
bound on $\mathcal{A}(Q)$ directly from
(\ref{eq:def-operator-A-without-kalman-gain}). The proof is written in a way
suitable for our generalization.

\begin{lemma}
If $R>0$ and $Q\geq0$, then
\begin{equation}
0\leq\mathcal{A}(Q)\leq Q, \label{eq:covariance-bound}%
\end{equation}
and
\begin{equation}
|\mathcal{A}(Q)|\leq|Q|. \label{eq:norm-of-AQ}%
\end{equation}

\end{lemma}

\begin{proof}
By the definition of operator $\mathcal{A}$ in
(\ref{eq:def-operator-A-without-kalman-gain}),
\[
\mathcal{A}\left(  Q\right)  =Q-QH^{\ast}(HQH^{\ast}+R)^{-1}HQ\leq Q,
\]
because $QH^{\ast}(HQH^{\ast}+R)^{-1}HQ\geq0$. It remains to show that
$\mathcal{A}(Q)\geq0$. Note that for any $A$, since $R>0$,
\[
A^{\ast}A+R\geq A^{\ast}A,
\]
and, consequently,
\[
\left(  A^{\ast}A+R\right)  ^{-1/2}A^{\ast}A\left(  A^{\ast}A+R\right)
^{-1/2}\leq I.
\]
Since for $B\geq0$, $B\leq I$ is the same as the spectral radius $\rho\left(
B\right)  \leq1$, and, for any $C$ and $D$, $\rho\left(  CD\right)
=\rho\left(  DC\right)  $, we have%
\begin{equation}
A\left(  A^{\ast}A+R\right)  ^{-1}A^{\ast}\leq I. \label{eq:est-inv}%
\end{equation}
Using (\ref{eq:est-inv}) with $A=Q^{1/2}H^{\ast}$ gives
\[
Q^{1/2}H^{\ast}(HQH^{\ast}+R)^{-1}HQ^{1/2}\leq I,
\]
and, consequently%
\[
QH^{\ast}(HQH^{\ast}+R)^{-1}HQ\leq Q,
\]
which gives $\mathcal{A}(Q)\geq0.$

Since $\mathcal{A}(Q)$ and $Q$ are symmetric, (\ref{eq:covariance-bound})
implies 
(\ref{eq:norm-of-AQ}).
\end{proof}

The pointwise continuity of operator $\mathcal{B}$ follows from Lemma
\ref{lem:continuity-K} as well.

\begin{lemma}
\label{lem:continuity-B}If $R>0$, and $P,Q\geq0$, then,%
\begin{align}
|\mathcal{B}(X,Q)-\mathcal{B}(Y,P)|  &  \leq|X-Y|\left(  1+c|Q|\right)  +\label{eq:difference-between-BX-and-BY}|Q-P||H||R^{-1}|(1+c|P|)(|d-HY|).
\end{align}

\end{lemma}

\begin{proof}
Estimating the difference, and using the pointwise bound on the Kalman gain
(\ref{eq:norm-kalman-gain}) and the pointwise continuity of the Kalman gain
from Lemma \ref{lem:continuity-K},
\begin{align*}
|\mathcal{B}(X,Q)&-\mathcal{B}(Y,P)|   
\\&=|(X+\mathcal{K}%
(Q)(d-HX))-(Y+\mathcal{K}(P)(d-HY))|\\
&=    |X-Y+\mathcal{K}(Q)(HY-HX)+(\mathcal{K}(Q)-\mathcal{K}(P))(d-HY)|\\
&  \leq|X-Y|+c|Q||X-Y|+|Q-P||H||R^{-1}|(1+c|P|)(|d-HY|)
\end{align*}
which is (\ref{eq:difference-between-BX-and-BY}).
\end{proof}

\subsection{$L^{p}$ bounds}

Using the pointwise estimate on the continuity of operator $\mathcal{A}$, we
can now estimate its continuity in $L^{p}$ spaces of random vectors. We will only need the result with one
of the arguments random and the other one constant (i.e., non random), which
simplifies its statement and proof. This is because the application of these
estimates will be the ensemble sample covariance, which is random, and the
state covariance, which is constant. In a similar way we will estimate the continuity of operator $\mathcal{B}$ in $L^p$.

\begin{lemma}
\label{lem:Lp-continuity-A}Let $Q$ be a random operator such that $Q\geq0$
almost surely (a.s.), let $P\geq0$ be constant, and let $R>0$. Then, for all
$1\leq p<\infty$,
\begin{equation}
\Vert\mathcal{A}(Q)-\mathcal{A}(P)\Vert_{p}\leq(1+c|P|)\Vert Q-P\Vert
_{p}+(c+c^2|P|)\Vert Q\Vert_{2p}\Vert Q-P\Vert_{2p}.
\label{eq:lp-difference-between-AQ-and-AP}%
\end{equation}

\end{lemma}

\emph{Proof}. From Lemma \ref{lem:continuity-A}, the triangle inequality,
Lemma \ref{lem:lp-cauchy}, and recognizing that $P$ is constant,
\begin{align*}
\Vert\mathcal{A}(Q)-\mathcal{A}(P)\Vert_{p}  &  \leq\Vert
|Q-P|(1+c|Q|+c|P|+c^{2}|P||Q|)\Vert_{p}\\
&  \leq\Vert Q-P\Vert_{p}+c\Vert Q\Vert_{2p}\Vert Q-P\Vert_{2p}\\
&  \qquad+c|P|\Vert Q-P \Vert_{p}+c^{2}|P|\Vert Q\Vert_{2p}\Vert
Q-P\Vert_{2p}\\
&  =(1+c|P|)\Vert Q-P\Vert_{p}+(c+c^2|P|)\Vert Q\Vert_{2p}\Vert Q-P\Vert_{2p}.
\qquad\endproof
\end{align*}

Instead of setting $P=0$ above, we get a better bound on $\Vert\mathcal{A}%
(Q)\Vert_{p}$ directly.

\begin{lemma}
Let $Q$ be a random operator such that $Q\geq0$ a.s., and $R>0$. Then, for all
$1\leq p<\infty$,
\begin{equation}
\Vert\mathcal{A}(Q)\Vert_{p}\leq\Vert Q\Vert_{p}. \label{eq:Lp-bound-A}%
\end{equation}

\end{lemma}

\begin{proof}
From (\ref{eq:norm-of-AQ}), it follows that $E\left(  \left\vert
\mathcal{A}(Q)\right\vert ^{p}\right)  \leq E\left(  \left\vert Q\right\vert
^{p}\right)  .$
\end{proof}

Using the pointwise estimate on the continuity of the operator $\mathcal{B}$,
we estimate its continuity in $L^{p}$ spaces of random vectors. Again, we keep the arguments of one of
the terms constant, which is all we will need, resulting in a simplification.

\begin{lemma}
\label{lem:Lp-continuity-B}Let $Q$ and $X$ be random operators, and $P$ and
$Y$ be constant. Let $Q\geq0$ a.s., $P\geq0$, and $R>0$. Then,
 for all
$1\leq p<\infty$,
\begin{align*}
\Vert\mathcal{B}(X,Q)-\mathcal{B}(Y,P)\Vert_{p}  &  \leq\Vert X-Y\Vert
_{p}+c\Vert Q\Vert_{2p}\Vert X-Y\Vert_{2p}\\
&  \qquad+\Vert Q-P\Vert_{p}|H||R^{-1}|(1+c|P|)(|d-HY|).
\end{align*}

\end{lemma}

\proof Applying the $L^{p}$ norm to the pointwise bound
(\ref{eq:difference-between-BX-and-BY}), using the triangle inequality,
recognizing that $P$ and $Y$ are constant, and applying the Cauchy-Schwarz
inequality (Lemma \ref{lem:lp-cauchy}), we get
\begin{align*}
\Vert\mathcal{B}(X,Q)   -\mathcal{B}(Y,P)\Vert_{p} &=\Vert|X-Y|\left(  1+c|Q|\right)  +|Q-P||H||R^{-1}%
|(1+c|P|)|d-HY|\Vert_{p}\\
&  \leq\Vert X-Y\Vert_{p}+c\Vert Q\Vert_{2p}\Vert X-Y\Vert
_{2p}\\
&\qquad+\Vert Q-P\Vert_{p}|H||R^{-1}|(1+c|P|)(|d-HY|).\qquad\endproof
\end{align*}

\section{$L^{p}$ laws of large numbers}

\label{sec:LpLLN}

Similarly as in \cite{LeGland-2011-LSA,Mandel-2011-CEK}, we will work with
convergence in all $L^{p}$, $1\leq p<\infty$. To prove $L^{p}$ convergence of
the initial ensemble mean and covariance to the mean and covariance of the
background distribution, we need the corresponding laws of large numbers.

\subsection{$L^{p}$ convergence of the sample mean}

The $L^{2}$ law of large numbers for $X_{1},\ldots,X_{N}\in L^{2}\left(
V\right)  $ i.i.d. is classical:%
\begin{equation}
\left\Vert E_{N}\left(  \boldsymbol{X}_{N}\right)  -E\left(  X_{1}\right)
\right\Vert _{2}\leq\frac{1}{\sqrt{N}}\left\Vert X_{1}-E\left(  X_{1}\right)
\right\Vert _{2}\leq\frac{2}{\sqrt{N}}\left\Vert X_{1}\right\Vert _{2}.
\label{eq:L2-LLN}%
\end{equation}
The proof relies on the expansion (assuming $E\left(  X_{1}\right)  =0$
without loss of generality),
\begin{align}
\left\Vert E_{N}\left(  \boldsymbol{X}_{N}\right)  \right\Vert _{2}^{2}  &
=E\left( \left\langle \frac{1}{N}\sum_{i=1}^{N}X_{i},\frac{1}{N}\sum_{i=1}^{N}%
X_{i}\right\rangle\right) =\frac{1}{N^{2}}\sum_{i=1}^{N}\sum_{j=1}^{N}E\left(
\left\langle X_{i},X_{j}\right\rangle \right) \label{eq:L2sum}\\
&  =\frac{1}{N^{2}}\sum_{i=1}^{N}E\left(  \left\langle X_{i},X_{i}%
\right\rangle \right)  =\frac{1}{N}E\left(  \left\langle X_{1},X_{1}%
\right\rangle \right)  =\frac{1}{N}\left\Vert X_{1}\right\Vert _{2}%
^{2},\nonumber
\end{align}
which yields $\left\Vert E_{N}\left(  \boldsymbol{X}_{N}\right)  \right\Vert
_{2}\leq\left\Vert X_{1}\right\Vert _{2}/\sqrt{N}$. To obtain $L^{p}$ laws of
large numbers for $p\neq2$, the equality (\ref{eq:L2sum}) needs to be replaced
by the following.

\begin{lemma}
(Marcinkiewicz-Zygmund inequality) If\/ $1\leq p<\infty$ and $Y_{i}\in
L^{p}\left(  V\right)  $, $E(Y_{i})=0$ and $E(|Y_{i}|^{p})<\infty$,
$i=1,...,N$, then
\begin{equation}
E\left(  \left\vert \sum_{i=1}^{N}Y_{i}\right\vert ^{p}\right)  \leq
B_{p}E\left(  \sum_{i=1}^{N}|Y_{i}|^{2}\right)  ^{p/2},
\label{eq:M-Z-inequality-Hilbert}%
\end{equation}
where $B_{p}$ depends on $p$ only.
\end{lemma}

\begin{proof}
For the finite dimensional case, see \cite{Marcinkiewicz-1937-FI} or \cite[p.
367]{Chow-1988-PT}. In the infinitely dimensional case, note that a separable
Hilbert space is a Banach space of Rademacher type 2 \cite[p. 159]%
{Araujo-1980-CLT}, which implies (\ref{eq:M-Z-inequality-Hilbert}) for any
$p\geq1$ \cite[Proposition 2.1]{Woyczynski-1980-MLL}. All infinitely
dimensional separable Hilbert spaces are isometric, so the same $B_{p}$ works
for any of them.
\end{proof}

The Marcinkiewicz-Zygmund inequality begets a variant of the weak law of large
numbers in $L^{p}$ norms, similarly as in \cite[Corollary 2, page
368]{Chow-1988-PT}. Note that the Marcinkiewicz-Zygmund inequality does not
hold in general Banach spaces, and in fact it characterizes Banach spaces of 
Rademacher type 2 \cite[Proposition 2.1]{Woyczynski-1980-MLL}, so it is important that $V=\mathbb{R}^{n}$
or $V$ is a separable Hilbert space, as assumed throughout.

\begin{theorem}
\label{lem:Lp-large-numbers} Let\/ $2\leq p<\infty$ and $X_{1},\ldots,X_{N}\in L^{p}\left(
V\right)  $ be i.i.d. Then,%
\begin{equation}
\left\Vert E_{N}\left(  \boldsymbol{X}_{N}\right)  -E\left(  X_{1}\right)
\right\Vert _{p}\leq\frac{C_{p}}{\sqrt{N}}\left\Vert X_{1}-E\left(
X_{1}\right)  \right\Vert _{p}\leq\frac{2C_{p}}{\sqrt{N}}\left\Vert
X_{1}\right\Vert _{p}, \label{eq:Lp-LLN}%
\end{equation}
where $C_{p}$ depends on $p$ only.
\end{theorem}

\begin{proof}
If $p=2$, the statement becomes (\ref{eq:L2-LLN}). Let $p>2$, and consider the case $E\left(  X_{1}\right)  =0$. By H\"{o}lder's inequality with
the conjugate exponents $\frac{p-2}{p}$ and $\frac{2}{p}$,%
\begin{align}
\sum_{i=1}^{N}\left\vert X_{i}\right\vert ^{2}  &  =\sum_{i=1}^{N}1\left(
\left\vert X_{i}\right\vert ^{2}\right)  \leq\left(  \sum_{i=1}^{N}\left(
1^{p/p-2}\right)  \right)  ^{(p-2)/p}\left(  \sum_{i=1}^{N}\left(  \left\vert
X_{i}\right\vert ^{2}\right)  ^{p/2}\right)  ^{2/p}\nonumber\\
&  =N^{\left(  p-2\right)  /p}\left(  \sum_{i=1}^{N}\left\vert X_{i}\right\vert ^{p}\right)  ^{2/p}
\label{eq:Lp-LLN-intermediate-step-1}%
\end{align}
Using the Marcinkiewicz-Zygmund inequality (\ref{eq:M-Z-inequality-Hilbert})
and (\ref{eq:Lp-LLN-intermediate-step-1}),%
\begin{align}
E\left(  \left\vert \sum_{i=1}^{N}X_{i}\right\vert ^{p}\right)   &  \leq
B_{p}E\left(  \left(  \sum_{i=1}^{N}\left\vert X_{i}\right\vert ^{2}\right)
^{p/2}\right) \nonumber\\
&  \leq B_{p}E\left(  \left(  N^{\left(  p-2\right)  /p}\left(  \sum_{i=1}%
^{N}\left\vert X_{i}\right\vert ^{p}\right)  ^{2/p}\right)  ^{p/2}\right)
\nonumber\\
&  \leq B_{p}N^{p/2-1}E\left(  \sum_{i=1}^{N}\left\vert X_{i}\right\vert
^{p}\right)  =B_{p}N^{p/2}E\left(  \left\vert X_{1}\right\vert ^{p}\right)  ,
\label{eq:Lp-LLN-intermediate-step-2}%
\end{align}
because $\frac{p}{2}\frac{p-2}{p}=\frac{p}{2}-1$, and the $X_{i}$ are
identically distributed. Taking the $p$-th root of both sides of
(\ref{eq:Lp-LLN-intermediate-step-2}) yields%
\[
\left\Vert \sum_{i=1}^{N}X_{i}\right\Vert _{p}\leq B_{p}^{1/p}N^{1/2}%
\left\Vert X_{1}\right\Vert _{p},
\]
and the first inequality in (\ref{eq:Lp-LLN}) follows after dividing by $N$. The general case when
$E\left(  X_{1}\right)  \neq0$ follows from the triangle inequality.
\end{proof}

\subsection{$L^{p}$ convergence of the sample covariance}

For the convergence of the ensemble covariance, we use the $L^{p}$ law of
large numbers in Hilbert-Schmidt norm, because $\mathrm{HS}\left(  V\right)  $
is a separable Hilbert space, while $\left[  V\right]  $ is not even a Hilbert
space. Convergence in the norm of $\left[  V\right]  $, the operator norm,
then follows from (\ref{eq:hilbert-schmidt-dominates}). See also \cite[Lemma
22.3]{LeGland-2011-LSA} for a related result using entry-by-entry estimates.

\begin{theorem}
\label{thm:LpLLN-sample-cov}Let $X_{1},\ldots,X_{N}\in L^{2p}\left(  V\right)
$ be i.i.d. and $p\geq2$. Then%
\begin{equation}
\left\Vert \left\vert C_{N}\left(  \boldsymbol{X}_{N}\right)
-\operatorname*{Cov}\left(  X_{1}\right)  \right\vert _{\mathrm{HS}%
}\right\Vert _{p}\leq\left(  \frac{2C_{p}}{\sqrt{N}}+\frac{4C_{2p}^{2}}%
{N}\right)  \left\Vert X_{1}\right\Vert _{2p}^{2}. \label{eq:LpLLN-sample-cov}%
\end{equation}
where $C_{p}$ is a constant which depends on $p$ only; in particular,
$C_{2}=1.$
\end{theorem}

\begin{proof}
Without loss of generality, let $E\left(  X_{1}\right)  =0$. Then%
\begin{align}
C_{N}\left(  \boldsymbol{X}_{N}\right)   &
-\operatorname*{Cov}\left(  X_{1}\right)  \label{eq:cov-err} \\
&  =\left(  E_{N}\left(  \boldsymbol{X}_{N}\otimes\boldsymbol{X}_{N}\right)
-E\left(  X_{1}\otimes X_{1}\right)  \right)  -E_{N}\left(  \boldsymbol{X}%
_{N}\right)  \otimes E_{N}\left(  \boldsymbol{X}_{N}\right)  ,\nonumber
\end{align}
where $\boldsymbol{X}_{N}\otimes\boldsymbol{X}_{N}= X_{1}\otimes
X_{1}+\ldots+X_{N}\otimes X_{N}  $. For the first term in
(\ref{eq:cov-err}), the $L^{p}$ law of large numbers (\ref{eq:Lp-LLN}) in
$\mathrm{HS}\left(  V\right)  $ yields%
\begin{align*}
E\left(  \left\vert E_{N}\left(  \boldsymbol{X}_{N}\otimes\boldsymbol{X}%
_{N}\right)  -E\left(  X_{1}\otimes X_{1}\right)  \right\vert _{\mathrm{HS}%
}^{p}\right)  ^{1/p} &  \leq\frac{2C_{p}}{\sqrt{N}}E\left(  \left\vert
X_{1}\otimes X_{1}\right\vert _{\mathrm{HS}}^{p}\right)  ^{1/p}\\
&  =\frac{2C_{p}}{\sqrt{N}}E\left(  \left\vert X_{1}\right\vert ^{2p}\right)
^{1/p}=\frac{2C_{p}}{\sqrt{N}}\left\Vert X_{1}\right\Vert _{2p}^{2},
\end{align*}
using (\ref{eq:hilbert-schmidt-tensor}). For the second term in
(\ref{eq:cov-err}), we use again (\ref{eq:hilbert-schmidt-tensor}) to get
\[
\left\vert E_{N}\left(  \boldsymbol{X}_{N}\right)  \otimes E_{N}\left(
\boldsymbol{X}_{N}\right)  \right\vert _{\mathrm{HS}}=\left\vert E_{N}\left(
\mathbf{X}_{N}\right)  \right\vert ^{2},
\]
and the $L^{p}$ law of large numbers in $V$ yields%
\begin{align*}
E\left(  \left\vert E_{N}\left(  \boldsymbol{X}_{N}\right)  \otimes
E_{N}\left(  \boldsymbol{X}_{N}\right)  \right\vert _{HS}^{p}\right)  ^{1/p}
&  =E\left(  \left\vert E_{N}\left(  \boldsymbol{X}_{N}\right)  \right\vert
^{2p}\right)  ^{1/p}\\
&  =\left\Vert E_{N}\left(  \boldsymbol{X}_{N}\right)  -0\right\Vert _{2p}%
^{2}\leq\left(  \frac{2C_{2p}}{\sqrt{N}}\left\Vert X_{1}\right\Vert
_{2p}\right)  ^{2}.
\end{align*}
It remains to use the triangle inequality for the general case.
\end{proof}

\section{Convergence of the unbiased squared root filter}

\label{sec:convergence}

By the law of large numbers, the sample mean and the sample covariance of the
initial ensemble converge to the mean and covariance of the background
distribution. Every analysis step is a continuous mapping of the mean and
covariance, and the convergence in the large ensemble limit follows. The
theorem below, which is the main result of this paper, quantifies this argument.

\begin{theorem}
\label{thm:lp-letkf-convergence} Assume that the state space $V$ is a finite
dimensional or separable Hilbert space, the initial state, denoted $X^{{(0)}}%
$, has a distribution on $V$ such that all moments exist, $E(|X^{(0)}%
|^{p})<\infty$ for every $1\leq p<\infty$, the initial ensemble
$\boldsymbol{X}_{N}^{\left(  0\right)  }$ is an i.i.d. sample from this
distribution, and the model is linear, $\mathcal{M}^{(k)}\left(  X\right)
=M^{(k)}X+b^{\left(  k\right)  }$ for every time $k$. 

Then, for all $k\geq 0$, the ensemble mean $\overline{\boldsymbol{X}}%
_{N}^{(k),\mathrm{a}}$ and covariance $Q_{N}^{(k),\mathrm{a}}$ from the
unbiased square root ensemble filter (Algorithm \ref{alg:srf}) converge to the
mean $\overline{X}^{(k),\mathrm{a}}$ and covariance $Q^{(k),\mathrm{a}}$ from
the Kalman filter (Algorithm \ref{alg:kf}), respectively, in all $L^{p}$, $1\leq p<\infty$, as
$N\rightarrow\infty$, with the convergence rate $1/\sqrt{N}$. Specifically,%
\begin{align}
\Vert\overline{\boldsymbol{X}}_{N}^{(k),\mathrm{a}}-\overline{X}%
^{(k),\mathrm{a}}\Vert_{p}  & \leq\frac{\operatorname*{const}\left(  p,k\right)
}{\sqrt{N}}%
,\label{eq:convergence-Lp-X}\\
\Vert Q_{N}^{(k),\mathrm{a}}-Q^{(k),\mathrm{a}}\Vert_{p}  & \leq\frac{\operatorname*{const}\left(  p,k\right)
}{\sqrt{N}}\label{eq:convergence-Lp-Q}%
\end{align}
for all $1\leq p<\infty$, and all $N=1,2,\ldots$, where we use $\operatorname*{const}\left(  p,k\right)  $ to denote a
generic constant which depends on $p$, $k$, on the norms of the various
constant (non-random) inputs and operators in the problem, and on the
background distribution, but not on the dimension of the state space or the
ensemble size. 
\end{theorem}

\proof We will proceed by induction on $k$.  For $k=0$ and $p\geq2$,
(\ref{eq:convergence-Lp-X}) and (\ref{eq:convergence-Lp-Q}) follow immediately
from the $L^{p}$ laws of large numbers, (\ref{eq:Lp-LLN}) and
(\ref{eq:LpLLN-sample-cov}), respectively. For $1\leq p<2,$ it is sufficient
to note that the $L^{p}$ norm is dominated by the $L^{2}$ norm.  


Let $k\geq1$ and assume that (\ref{eq:convergence-Lp-X}) and
(\ref{eq:convergence-Lp-Q}) hold with $k-1$ in place of $k$, for all $N>0$,
and all $1\leq p<\infty$. From (\ref{eq:filtering-forecast-mean}),
(\ref{eq:filtering-forecast-covariance}), and
(\ref{eq:forecast-ens-covariance}), it follows that
\begin{align}
|\overline{X}^{(k),\mathrm{f}}|,|Q^{(k),\mathrm{f}}|\le \operatorname*{const}\left(k\right),
\quad \|Q_{N}^{(k),\mathrm{f}}\|_p\leq \operatorname*{const}\left(p,k\right),\label{eq:a-priori-bounds}
\end{align}
for all $1\leq p<\infty$.
Comparing (\ref{eq:filtering-forecast-covariance})
and (\ref{eq:forecast-ens-covariance}), we have
\begin{align}
\Vert Q_{N}^{(k),\mathrm{f}}-Q^{(k),\mathrm{f}}\Vert_{p} &  =\Vert M^{\left(
k\right)  }Q_{N}^{(k-1),\mathrm{a}}M^{\left(  k\right)  \ast}-M^{\left(
k\right)  }Q^{(k-1),\mathrm{a}}M^{\left(  k\right)  \ast}\Vert_{p}\nonumber\\
&  \leq|M^{\left(  k\right)  }|^{2}\Vert Q_{N}^{(k-1),\mathrm{a}%
}-Q^{(k-1),\mathrm{a}}\Vert_{p},\label{eq:forecast-cov-diff}%
\end{align}
and from Lemma \ref{lem:Lp-continuity-A},
\begin{align}
\Vert Q_{N}^{(k),\mathrm{a}} &  -Q^{(k),\mathrm{a}}\Vert_{p}=\Vert
\mathcal{A}(Q_{N}^{(k),\mathrm{f}})-\mathcal{A}(Q^{(k),\mathrm{a}})\Vert
_{p}\nonumber\\
\leq &  (1+c^{\left(  k\right)  }|Q^{(k),\mathrm{f}}|)\Vert Q_{N}^{(k),\mathrm{f}}%
-Q^{(k),\mathrm{f}}\Vert_{p}\label{eq:QN-Q-diff}\\
&  +(c^{\left(  k\right)  }+(c^{\left(  k\right)  })^2)|Q^{(k),\mathrm{f}}|)\Vert Q_{N}^{(k),\mathrm{f}}\Vert_{2p}\Vert
Q_{N}^{(k),\mathrm{f}}-Q^{(k),\mathrm{f}}\Vert_{2p},\nonumber
\end{align}
where $c^{\left(  k\right)  }=|H^{\left(  k\right)  }|^{2}|(R^{\left(
k\right)})^{  -1}|$. Combining (\ref{eq:a-priori-bounds})--(\ref{eq:QN-Q-diff})
and using the induction assumption (\ref{eq:convergence-Lp-Q}), with $2p$ in
place of $p$ when necessary, we obtain%
\[
\Vert Q_{N}^{(k),\mathrm{a}}-Q^{(k),\mathrm{a}}\Vert_{p}\leq\frac
{\operatorname*{const}\left(  p,k\right)  }{\sqrt{N}}.
\]

For the convergence of the ensemble mean to the Kalman filter mean, we have
\begin{align*}
\Vert\overline{\boldsymbol{X}}_{N}^{(k),\mathrm{f}}-\overline{X}^{(k),\mathrm{f}}%
\Vert_{p}=\Vert M^{\left(  k\right)  }(\overline{\boldsymbol{X}}_{N}^{(k-1),\mathrm{a}}%
-\overline{X}^{(k-1),\mathrm{a}})\Vert_{p}\leq|M^{\left(  k\right)  }|\Vert
\overline{\boldsymbol{X}}_{N}^{(k-1),\mathrm{a}}-\overline{X}^{(k-1),\mathrm{a}}\Vert_{p}%
,\label{eq:forecast-mean-diff}%
\end{align*}
from Lemma \ref{lem:srf-mean-covariance-evolution}. By Lemma
\ref{lem:Lp-continuity-B},%
\begin{align*}
\Vert\overline{\boldsymbol{X}}_{N}^{(k),\mathrm{a}}-\overline{X}%
^{(k),\mathrm{a}}\Vert_{p} &  =\Vert\mathcal{B}(\overline{\boldsymbol{X}}%
_{N}^{(k),\mathrm{f}},Q_{N}^{(k),\mathrm{f}})-\mathcal{B}(\overline
{X}^{\left(  k\right),  \mathrm{f}},Q^{(k),\mathrm{f}})\Vert_{p}\\
\leq &  \Vert\overline{\boldsymbol{X}}_{N}^{(k),\mathrm{f}}-\overline
{X}^{\left(  k\right),  \mathrm{f}}\Vert_{p}+c^{(k)}\Vert
Q_{N}^{(k),\mathrm{f}}\Vert_{2p}\Vert\overline{\boldsymbol{X}}_{N}%
^{(k),\mathrm{f}}-\overline{X}^{\left(  k\right),  \mathrm{f}}\Vert_{2p}\\
&  +\Vert Q_{N}^{(k),\mathrm{f}}-Q^{(k),\mathrm{f}}\Vert_{p}|H||R^{-1}%
|(1+c^{(k)}|Q^{(k),\mathrm{f}}|)(|d-H\overline{X}^{\left(  k\right)
,\mathrm{f}}|),
\end{align*}
which, together with (\ref{eq:a-priori-bounds}) and
(\ref{eq:forecast-cov-diff}), and the induction assumption (\ref{eq:convergence-Lp-Q}), with $2p$ in
place of $p$ when necessary, yields
\[
\Vert\overline{\boldsymbol{X}}_{N}^{(k),\mathrm{a}}-\overline{X}%
^{(k),\mathrm{a}}\Vert_{p}\leq\frac{\operatorname*{const}\left(  p,k\right)
}{\sqrt{N}}.\qquad\endproof
\]

\bibliographystyle{siam}
\bibliography{../../references/geo,../../references/other}

\end{document}